\documentclass[a4paper,11pt]{amsart}
\usepackage{etex}
\usepackage{tikz, tikz-cd}

\usepackage{bm}
\usepackage{amssymb}
\usepackage{mathrsfs}
\usepackage{amsmath, amssymb,amsthm,latexsym,amscd,mathrsfs}
\usepackage{indentfirst}
\usepackage{stmaryrd}
\usepackage{graphicx}
\usepackage{subfigure}
\usepackage{extarrows}
\usepackage{amssymb}
\usepackage{amsmath,amssymb,amsthm,latexsym,amscd,mathrsfs}
\usepackage{indentfirst}
\usepackage{multirow}
\usepackage{latexsym}
\usepackage{amsfonts}
\usepackage{color}
\usepackage{pictexwd,dcpic}
\usepackage{graphicx}
\usepackage{psfrag}
\usepackage{hyperref}
\usepackage{comment}
\usepackage{cases}
\usepackage{dutchcal}
\usepackage{arydshln}
\usepackage{tikz-cd}
\usepackage{amsmath, amssymb}

\makeatletter
\@namedef{subjclassname@2010}{{\rm2020} Mathematics Subject Classification}
\makeatother

\newcommand{\ROM}[1]{\mathrm{\uppercase\expandafter{\romannumeral#1}}}
\theoremstyle{definition}
\numberwithin{equation}{section}
\theoremstyle{plain}

\newtheorem{thm}{Theorem}[section]
\newtheorem{lem}{Lemma}[section]

\newtheorem{cor}{Corollary}[section]

\newtheorem{rem}{Remark}[section]
\newtheorem{prop}{Proposition}[section]

\newtheorem{ack}{Acknowledgements}   
\makeatletter

\allowdisplaybreaks
\makeatother
\title[Non-extendability of complex structures
]{Non-extendability of complex structures}
\author[Z. Z. Tang]{Zizhou Tang}\address{Chern Institute of Mathematics $\&$ LPMC, Nankai University, Tianjin 300071, P. R. China}
\email{zztang@nankai.edu.cn}
\author[W. J. Yan]{Wenjiao Yan}\address{School of Mathematical Sciences, Laboratory of Mathematics and Complex Systems, Beijing Normal University, Beijing, 100875, P. R. China}
\email{wjyan@bnu.edu.cn}
\thanks {The project is partially supported by the NSFC (No. 12271038, 12371048), and the Nankai Zhide Foundation. 
}


\begin{document}
	
	\maketitle
	\begin{abstract}
		There exists a complex structure $J$ on a connected open subset $S^3_{\delta}\times S^3$ of $S^6$. The present paper proves that: (1) 
		$J$ can be extended to a global almost complex structure $\widetilde{J}$ on $S^6$; (2) any extension to $S^6$ is necessarily non-integrable. Therefore, it is impossible to deform $\widetilde{J}$ to an integrable almost complex structure on $S^6$ while fixing it on $S^3_{\delta}\times S^3$. This phenomenon indicates that the deformation strategy suggested by S.-T. Yau in his Problem 52 cannot be realized in this sense.

	\end{abstract}
	
	\section{\textbf{Introduction}}\label{sec1}
	
	A real smooth manifold is called a complex manifold if it admits a complex structure with holomorphic coordinate charts. One of the most important problems in modern pure mathematics is the existence or non-existence of complex structures. While every orientable surface admits a complex structure, explicit examples in higher dimensions remain comparatively scarce. Classical examples include the Hopf manifolds, whose underlying smooth manifolds are $S^1 \times S^{2n+1}$; the Calabi-Eckmann manifolds (\cite{CE53}), whose underlying smooth manifolds are $S^{2p+1} \times S^{2q+1} $ $(p,q\geq 1)$; products of odd-dimensional Brieskorn homotopy spheres $\Sigma^{2p+1} \times \Sigma^{2q+1} (p,q\geq 1)$ (\cite{BV68}); twistor spaces of Taubes' self-dual $4$-manifolds; and complex projective manifolds, among others.

	Every complex manifold carries a canonical almost complex structure induced by local holomorphic coordinates,
    and is therefore an almost complex manifold. Recall that \textit{an almost complex structure}
	on a smooth manifold $M$ is an endomorphism
	$J: TM\rightarrow TM$ of the tangent bundle satisfying $J^2=-\mathrm{Id}$.
	Such a structure forces $M$ to be
	even-dimensional and oriented. A smooth manifold $M^{2n}$ admits an almost complex structure if and only if the structure group of its tangent bundle can be reduced from $O(2n)$ to $U(n)$. 
	An almost complex structure is called \textit{integrable} (or simply \emph{a}\emph{ complex structure}) if it is the canonical almost complex structure of a complex manifold.

	The converse question, whether every almost complex manifold admits an integrable almost complex structure, is substantially more subtle. In 1951, Steenrod conjectured in his seminal monograph \cite{Ste51}:
	``It seems highly unlikely that every almost complex manifold has a complex analytic structure." This intuition was confirmed in dimension $4$: in 1966, Van de Ven \cite{VdV66} constructed the first example of a compact almost complex $4$-manifold that admits no complex structures, leveraging Chern numbers and Milnor's results. 
	Yau \cite{Yau76} later constructed the first compact parallelizable $4$-manifold without any complex structure, using Massey products. 
	Historically, essentially all classical obstructions to the existence of complex structures occur in real dimension $4$, where the classification of compact complex surfaces is well established.

	In higher dimensions, determining whether an almost complex structure can be deformed to an integrable one remains extremely challenging. For open manifolds, however, powerful results are available. A theorem of Gromov \cite{Gro73}, refined by Landweber \cite{Lan74}, states:
	
	\vspace{2mm}
	\noindent
	\textbf{Theorem (Gromov-Landweber).}
	\emph{Let $V^{2n}$ be a $2n$-dimensional open manifold with $H^i(V, \mathbb{Z}) = 0$ for any $i > n$.
		Then every almost complex structure on $V$ is homotopic (deformed) to an integrable one.}
	\vspace{2mm}

	Combining Gromov's work in dimension $4$ with a result of Adachi \cite{Ada79} in dimension $6$ yields:
	\vspace{2mm}
	
	\noindent
	\textbf{Theorem (Gromov-Adachi).} \emph{Let $V^{2n}$ be an open manifold of dimension $2n\leq 6$. Then every almost complex structure on $V$ is homotopic (deformed) to an integrable one.}
	\vspace{2mm}

	For compact manifolds, no comparable criterion is known. A celebrated open problem is the Hopf problem, which asks whether $S^6$ admits a complex structure. The $6$-sphere $S^6$ carries a canonical almost complex structure arising from its realization as the unit sphere in the imaginary octonions $\text{Im} (\mathbb{O})$ (\cite{Kir47}). However, this almost complex structure is non-integrable~(\cite{EF51, EL51}).
	
	One of the most intriguing proposals to address this question is the Problem 52, posed by S.-T. Yau in \cite{SY94} (and on multiple subsequent occasions) as one of the most profound open problems in geometry:
	
	\vspace{2mm}
	
	\noindent
	\textbf{Yau's Problem 52.} \emph{Prove that every compact almost complex manifold with dimension $\geq 6$  admits an integrable almost complex structure. This question is still unsolved for the $6$-dimensional sphere. One approach is to form a parabolic flow in the space of almost complex structures to deform an almost complex structure to an integrable one}. 
\vspace{2mm}

Recent progress has emerged from isoparametric theory: the authors \cite{TY22}  explicitly constructed complex structures on certain isoparametric hypersurfaces in spheres of arbitrarily large dimensions, including explicit examples such as $S^1\times S^3\times S^2$ (different from the product of complex structures) and $S^1\times S^7\times S^6$. More recently, \cite{QTY25} found additional families of complex manifolds arising from isoparametric foliations in unit spheres,  further enriching the landscape of higher-dimensional complex manifolds.


Recall that Calabi and Eckmann \cite{CE53} constructed a family of complex manifolds homeomorphic to
$S^{2p+1}\times S^{2q+1}~(p, q\geq 1)$. For the purpose of this paper, we focus on the case $S^3\times S^3$. Let $\{e_1, e_2, e_3\}$ and $\{e_4, e_5, e_6\}$ be global orthonormal tangent frames on the first and second $S^3$ factors, respectively. Furthermore, we choose $e_1$ and $e_4$ to be tangent to the fibers of the standand Hopf fibration $S^1\hookrightarrow S^3\rightarrow S^2$ on each factor, respectively.
Consider the almost complex structure $J$ essentially defined by Calabi and Eckmann, which pairs the vertical directions $e_1$ with $e_4$:
\begin{equation*}
	\begin{aligned}
		&Je_1=e_4, &Je_2&=e_3, &Je_5&=e_6\\
		&Je_4=-e_1, &Je_3&=-e_2,& Je_6&=-e_5.
	\end{aligned}
\end{equation*}
That is,
\begin{equation}\label{J}
	J(e_1, e_2, e_3, e_4, e_5, e_6) = (e_1, e_2, e_3, e_4, e_5, e_6) J_0,
\end{equation}
where $J_0$ is the block matrix:
\begin{equation*}\label{J0}
	\small	J_0=\left(
	\begin{array}{cccc:cc}
		&   &   & -1 & &   \\
		&   & -1 &   &  &   \\
		& 1 &   &   & &   \\
		1 &   &   &   & &   \\
		\cdashline{1-6}
		&   &   &   &   & -1 \\
		&   &   &   & 1 &
	\end{array}
	\right).
\end{equation*}

For $\delta\in (0, 1)$, define an open subset $S^3_{\delta}\subset S^3$ by
\begin{equation}\label{S^3}
	S^3_{\delta}:=\{(x_1, x_2, x_3, x_4)\in S^3~|~x_1< \delta\}\subset S^3.
\end{equation}
The product $S^3_{\delta}\times S^3$ naturally inherits the Calabi-Eckmann complex structure defined in \eqref{J}, which we denote by the same symbol $J$.

Notably, $S^3_{\delta}\times S^3$ is homeomorphic to $\mathbb{R}^3\times S^3$. In what follows, we will embed $\mathbb{R}^3\times S^3$ into $S^6$ as an open subset.



As the main result of this paper, we prove:

\begin{thm}\label{thm}
	There exists a complex structure $J$ on the connected open subset $S^3_{\delta}\times S^3$ of $S^6$. 
	This structure $J$ admits an extension to a global almost complex structure $\widetilde{J}$ on $S^6$. However, any extension is non-integrable. Therefore, it is impossible to deform $\widetilde{J}$ to an integrable almost complex structure on $S^6$ while fixing it on $S^3_{\delta}\times S^3$.
\end{thm}

\begin{rem}\label{rem1.1}
	This theorem indicates that the deformation strategy suggested by S.-T.Yau in his Problem 52 cannot be realized in this sense. In fact, as will be detailed in Remark \ref{rem3.1}, one can replace $S^3_{\delta}\times S^3$ with a connected open subset of $S^6$ of arbitrarily small measure while maintaining the validity of Theorem \ref{thm}.
	
	Regarding the deformation theory, Lemma 6.1.2 in \cite{Huy05} identifies the integrability condition for a deformed almost complex structure with the Maurer-Cartan equation. Consequently, Theorem \ref{thm} implies that the associated deformation problem---governed by the Maurer-Cartan equation subject to the initial condition on $S^3_{\delta}\times S^3$---admits no solution.
	
\end{rem}


We emphasize that the extendability of $J$ to a global almost complex structure on $S^6$
in Theorem~\ref{thm} relies on obstruction-theoretic arguments, and such extendability is highly exceptional. For instance, 
consider $\mathbb{C}P^2$ with its standard complex structure (inducing the positive orientation). Let $D^4 \subset \mathbb{C}P^2$ be a smoothly embedded closed 4-disc and set $M := \mathbb{C}P^2 \setminus D^4.$ 
Clearly, the open manifold $M$ inherits a positive orientation from $\mathbb{C}P^2$. However, we will show 


\begin{prop}\label{prop}
	The open manifold $M = \mathbb{C}P^2 \setminus D^4$ admits an integrable almost complex structure $J$ which induces the negative orientation. Nevertheless, $J$ cannot be extended to a global almost complex structure on $\mathbb{C}P^2$.
\end{prop}

The paper is organized as follows.
In Section 2, we identify
$S^3_{\delta} \times S^3$ with an open subset of $S^6$ and prove the extendability of $J$ to a global almost complex structure on $S^6$. In Section 3, we establish that no extension of $J$ can be integrable on $S^6$. Finally, Section 4 presents a proof of Proposition \ref{prop}.

\section{\textbf{Extendability of $J$ to a global almost complex structure on $S^6$}}

We begin with a decomposition of $S^6$ into disc bundles. This idea originates from the isoparametric theory in unit spheres.  
E. Cartan initiated the study of isoparametric hypersurfaces in the 1930s; he characterized 
isoparametric hypersurfaces in spheres as those with constant principal curvatures. Subsequently, M\"unzner proved that the number $g$ of distinct principal curvatures can only be
$1,2,3,4$, or $6$. 
Moreover, there are two focal submanifolds in each isoparametric family, which are minimal submanifolds of the sphere. Every isoparametric hypersurface divides the ambient sphere into two disc bundles of constant radii over the two focal submanifolds; the hypersurface itself is their common boundary (cf. \cite{Mun80}).


For $g=2$, an isoparametric hypersurface in $S^{n+1}$ is diffeomorphic to $S^p\times S^q$ with $p+q=n$, and its two focal submanifolds are precisely $S^p$ and $S^q$. Moreover, the normal bundles of the focal submanifolds are trivial. When $p=2$ and $q=3$, each isoparametric hypersurface in $S^6$ is diffeomorphic to $S^2\times S^3$, and the focal submanifolds are $S^2$ and $S^3$, respectively. The corresponding open disc bundles are diffeomorphic to $\mathbb{R}^4\times S^2$ and $\mathbb{R}^3\times S^3$. Consequently, $S^3_{\delta} \times S^3\cong \mathbb{R}^3\times S^3$ is diffeomorphic to an open subset of $S^6$. More precisely, the two open disc bundles $\mathbb{R}^4\times S^2$ and $\mathbb{R}^3\times S^3$ are embedded in $S^6$ via the maps
\begin{alignat*}{2}
	S^2 \times \mathbb{R}^4&\to S^6, &\qquad\quad  &\mathbb{R}^3 \times S^3 \to S^6 \\
	(x,~y) &\mapsto \frac{(x,y)}{\sqrt{1+|y|^2}},&\qquad\quad 	&\quad(z,~w) \mapsto \frac{(z,w)}{\sqrt{1+|z|^2}} .
\end{alignat*}
In the intersection of the two images, we have $\frac{(z,w)}{\sqrt{1+|z|^2}} = \frac{(x,y)}{\sqrt{1+|y|^2}}$, and a direct comparison gives
$z = \frac{x}{|y|}$ and $w = \frac{y}{|y|}$.  Thus we obtain the decomposition of $S^6$ into two disc bundles:
\begin{equation}\label{isop}
	S^6\cong S^2\times \mathbb{R}^4\sqcup_{\varphi} \mathbb{R}^3\times S^3,
\end{equation}
where the gluing diffeomorphism $\varphi$ acts on the common region by
\begin{equation}\label{varphi}
	\begin{aligned}
		\varphi: S^2\times \left(\mathbb{R}^4\backslash \{0\}\right)&\rightarrow \left(\mathbb{R}^3\backslash \{0\}\right)\times S^3\\
		(x, ~~y)\quad &\mapsto \quad (\frac{x}{|y|}, ~~\frac{y}{|y|}).
	\end{aligned}
\end{equation}

Recall that for $\delta\in (0, 1)$, $S^3_{\delta}:=\{(x_1, x_2, x_3, x_4)\in S^3~|~x_1< \delta\}$. Since $S^3_{\delta}\subset S^3\backslash \{N\}$, where $N=(1, 0, 0, 0)$ is the north pole, we may identify $S^3_{\delta}$ with an open disc $\mathbb{R}^3_{\delta}$ of radius $\sqrt{\frac{1 + \delta}{1 - \delta}}$ in $\mathbb{R}^3$ via the stereographic projection
\begin{equation}\label{stereo}
	\begin{aligned}
		\phi_N: \mathbb{R}^3 & \rightarrow S^3 \setminus \{N\}\\
		\ z & \mapsto u =\frac{1}{1 + |z|^2}(|z|^2-1, 2z).
	\end{aligned}
\end{equation}
Explicitly,
$$S^3_{\delta}\cong\mathbb{R}^3_{\delta} := \left\{ z \in \mathbb{R}^3 ~\middle|~ |z|^2 < \frac{1 + \delta}{1 - \delta} \right\} \subset \mathbb{R}^3.$$
Using the diffeomorphism $\varphi$ in \eqref{varphi}, we identify $\left(\mathbb{R}^3_{\delta}\backslash\{0\}\right)\times S^3$ with
$S^2\times \mathbb{R}^4_{\delta}$, where
$$\mathbb{R}_\delta^4 = \left\{ y \in \mathbb{R}^4 ~\middle|~ |y|^2 > \frac{1-\delta}{1+\delta} \right\},$$
the complement of a closed disc of radius $\sqrt{\frac{1-\delta}{1+\delta}}$ in $\mathbb{R}^4$. 

Via the diffeomorphism $\varphi$, the almost complex structure $J$ on $\left(S^3_{\delta}\backslash\{S\}\right)\times S^3$ can be transferred to $S^2\times \mathbb{R}^4_{\delta}$, where $S=(-1, 0,0,0)$ is the south pole of $S^3$; we keep the notation $J$ for the transferred structure. Consequently, extending the almost complex structure from $S^3_{\delta}\times S^3$ to $S^6$ reduces to extending $J$ from $S^2\times\mathbb{R}^4_{\delta}$ to $S^2\times\mathbb{R}^4$.


As the first step to extend the almost complex structure $J$ defined in \eqref{J}, we explicitly construct a global orthonormal tangent frame $\{e_1, \ldots, e_6\}$ of $S_\delta^3 \times S^3$ along with its local coordinate representation, as mentioned earlier in the introduction.

Identify $S^3\subset \mathbb{R}^4\cong\mathbb{H}$ with the unit quaternions. Let $\{\textbf{1, i, j, k}\}$ be the standard basis of $\mathbb{H}$. At $(u, v)\in S_\delta^3 \times S^3$ with $u=(u_1, u_2, u_3, u_4), v=(v_1, v_2, v_3, v_4)$, we choose $\{e_1, \ldots, e_6\}$ as follows  (where we denote the zero vector in $\mathbb{R}^4$ by $\textbf{0}=(0,0,0,0)$):
$$
\begin{aligned}
	e_1&=(\text{\textbf{i}}u, \textbf{0})=\Big( (-u_2, u_1, -u_4, u_3), \textbf{0} \Big), &e_2=(\text{\textbf{j}}u, \textbf{0})= \Big( (-u_3, u_4, u_1, -u_2), \textbf{0}\Big),\\
	e_3&=(\text{\textbf{k}}u, \textbf{0})=\Big( (-u_4, -u_3, u_2, u_1), \textbf{0}\Big),
	&e_4=(\textbf{0}, \text{\textbf{i}}v)=\Big(\textbf{0}, ~~(-v_2, v_1, -v_4, v_3)\Big), \\
	e_5&=(\textbf{0}, \text{\textbf{j}}v)= \Big( \textbf{0}, ~~(-v_3, v_4, v_1, -v_2)\Big), &  e_6=(\textbf{0}, \text{\textbf{k}}v)=\Big( \textbf{0}, ~~(-v_4, -v_3, v_2, v_1)\Big).
\end{aligned}
$$

We now express $\{e_1,\ldots, e_6\}$ in suitable local coordinates.
Combining $\varphi$ with $\phi_N$, we obtain a parametrization
\begin{alignat}{3}\label{phi}
	S^2 \times \mathbb{R}_\delta^4 &\overset{\varphi}{\xrightarrow{\qquad}} &\left(\mathbb{R}_\delta^3 \backslash\{0\}\right)\times S^3 &\overset{(\phi_N, Id)}{\xrightarrow{\qquad}} \left(S_\delta^3\backslash\{S\}\right) \times S^3 \\
	\qquad (x, ~y)~~ &\quad\mapsto&\left( \frac{x}{|y|}, ~\frac{y}{|y|} \right) &\quad\mapsto \quad(u,v):= \left(\Big( \frac{1-|y|^2}{1+|y|^2},~ \frac{2|y|x}{1+|y|^2}\Big),~ \frac{y}{|y|} \right)\nonumber
\end{alignat}
We further parametrize $S^2$ by stereographic projection from the south pole,
\begin{equation}\label{phi_S}
	\begin{aligned}
		\phi_S:	~~\mathbb{R}^2\quad&\rightarrow \quad S^2\backslash\{S\}\\
		t=(t_1, t_2)~~&\mapsto ~~x= (x_1, x_2, x_3):=\left( \frac{1-|t|^2}{1+|t|^2},~~ \frac{2t}{1+|t|^2} \right).
	\end{aligned}
\end{equation}
Together with \eqref{phi}, we obtain a parametrization:
\begin{equation}\label{psi}
	\begin{aligned}
		\psi: ~~\mathbb{R}^2 \times \mathbb{R}_\delta^4 &\longrightarrow S_\delta^3 \times S^3 \\
		(t, ~y) &\mapsto (u, v),
	\end{aligned}	
\end{equation}
where $$u = \left( \frac{1-|y|^2}{1+|y|^2}, ~~\frac{2|y|}{1+|y|^2} \Big( \frac{1-|t|^2}{1+|t|^2}, \frac{2t}{1+|t|^2} \Big) \right) \quad\text{and}\quad v = \frac{y}{|y|}.$$

In these local coordinates, the orthonormal basis $\{e_1, \ldots, e_6\}$
is expressed as follows: 
$$
\begin{aligned}
	e_1&=\left(\frac{1}{1 + |y|^2}\left( -2|y|x_1, ~1-|y|^2,~-2|y| x_3,~2|y| x_2 \right), ~~\textbf{0}\right),\\
	e_2&=\left(\frac{1}{1 + |y|^2}\left( -2|y| x_2, ~2|y| x_3, ~1-|y|^2, ~-2|y| x_1 \right), ~~\textbf{0}\right),\\
	e_3&=\left( \frac{1}{1 + |y|^2}\left( -2|y| x_3, ~-2|y| x_2, ~2|y| x_1, ~1- |y|^2  \right), ~~\textbf{0}\right),\\
	e_4&= \left( \textbf{0}, ~~\frac{1}{|y|}(-y_2, y_1, -y_4, y_3)\right),\\
	e_5&=\left( \textbf{0},~~\frac{1}{|y|}(-y_3, y_4, y_1, -y_2)\right),\\
	e_6&= \left( \textbf{0}, ~~ \frac{1}{|y|}(-y_4, -y_3, y_2, y_1)\right),
\end{aligned}
$$
with $x=(x_1, x_2, x_3)$ given by \eqref{phi_S}, $y=(y_1, y_2, y_3, y_4)\in \mathbb{R}^4_{\delta}$.

On the other hand, the natural coordinate vector fields $\{\frac{\partial}{\partial t_1}, \frac{\partial}{\partial t_2}, \frac{\partial}{\partial y_1}, $ $\frac{\partial}{\partial y_2}, \frac{\partial}{\partial y_3}, \frac{\partial}{\partial y_4}\}$ of $\left(S_\delta^3\backslash\{S\}\right) \times S^3$ at $(u, v)=\psi(t, y)$ can be computed explicitly:
\begin{equation*}\label{natural basis}
	\begin{aligned}
		\frac{\partial}{\partial t_1} = \psi_{*} \left( \frac{\partial}{\partial t_1} \right) 
		&= \left( \Big(0, \frac{4|y|}{(1+|y|^2)(1+|t|^2)^2}\left( -2t_1,1+t_2^2-t_1^2, -2t_1t_2\right)\Big), ~~\textbf{0} ~~~\right),\\
		\frac{\partial}{\partial t_2} = \psi_{*} \left( \frac{\partial}{\partial t_2} \right)&=\left( \Big(0, \frac{4|y|}{(1+|y|^2)(1+|t|^2)^2}\left( -2t_2,~ -2t_1t_2,~1+t_1^2-t_1^2\right)\Big), ~~\textbf{0}~~~\right),\\
		\frac{\partial}{\partial y_1} =\psi_{*} \left( \frac{\partial}{\partial y_1} \right) &= \left( \frac{2y_1}{(1 + |y|^2)^2}\Big(-2, \frac{1 - |y|^2}{|y|} x \Big), ~~\frac{1}{|y|^3} \Big(|y|^2 - y_1^2, -y_1 y_2, -y_1 y_3, -y_1 y_4\Big) \right),\\
		\frac{\partial}{\partial y_2} =\psi_{*} \left( \frac{\partial}{\partial y_2} \right) &= \left(\frac{2y_2}{(1 + |y|^2)^2}\Big(-2, \frac{1 - |y|^2}{|y| } x \Big),~~ \frac{1}{|y|^3} \Big(-y_1 y_2, |y|^2 - y_2^2, -y_2 y_3, -y_2 y_4\Big) \right),\\
		\frac{\partial}{\partial y_3} =\psi_{*} \left( \frac{\partial}{\partial y_3} \right) &= \left(\frac{2y_3}{(1 + |y|^2)^2}\Big(-2, \frac{1 - |y|^2}{|y| } x \Big),~~ \frac{1}{|y|^3} \Big(-y_1 y_3, -y_2 y_3, |y|^2 - y_3^2, -y_3 y_4\Big) \right),\\
		\frac{\partial}{\partial y_4} =\psi_{*} \left( \frac{\partial}{\partial y_4} \right) &=\left(\frac{2y_4}{(1 + |y|^2)^2}\Big(-2, \frac{1 - |y|^2}{|y| } x \Big),~~ \frac{1}{|y|^3}\Big(-y_1 y_4, -y_2 y_4, -y_3 y_4, |y|^2 - y_4^2\Big) \right),
	\end{aligned}
\end{equation*}
where again $x$ is understood via \eqref{phi_S}.

Expressing the coordinate vector fields $\{ \frac{\partial}{\partial t_1} , \frac{\partial}{\partial t_2}, \frac{\partial}{\partial y_1},\ldots, \frac{\partial}{\partial y_4} \}$ in terms of the orthonormal frame $\{e_1,\ldots, e_6\}$, we write
\begin{equation}\label{G}
	( \frac{\partial}{\partial t_1} , \frac{\partial}{\partial t_2}, \frac{\partial}{\partial y_1},\frac{\partial}{\partial y_2}, \frac{\partial}{\partial y_3}, \frac{\partial}{\partial y_4}) = (e_1, e_2, e_3, e_4, e_5, e_6) G,
\end{equation}
where $G$ is a non-singular matrix of block form whose entries are functions:
\begin{equation}\label{G}
	G = \begin{pmatrix} A_{3\times 2} & B_{3\times 4} \\ \textbf{0}_{3\times 2} & C_{3\times 4} \end{pmatrix}.
\end{equation}
The entries of $A, B, C$ are given explicitly by
$$
A = \frac{4|y|}{(1+|y|^2)^2 (1+|t|^2)^2}\left( (1-|y|^2)\small\begin{pmatrix}
	-2t_1& -2t_2\\
	1-t_1^2+t_2^2& -2t_1t_2 \\
	-2t_1t_2 &	1+t_1^2 -t_2^2
\end{pmatrix} + 2|y| \small\begin{pmatrix}
	-2t_2 & 2t_1 \\
	-2t_1t_2 & -1+t_1^2-t_2^2 \\
	1+t_1^2-t_2^2& 2t_1t_2
\end{pmatrix}\right),
$$
$$
\begin{aligned}
	B& = \frac{2}{|y| (1+|y|^2)(1+|t|^2)} \begin{pmatrix}
		1-|t|^2 \\
		2t_1 \\
		2t_2
	\end{pmatrix} \Big(y_1, y_2, y_3, y_4\Big),\\C &= \frac{1}{|y|^2} \begin{pmatrix} \text{\textbf{i}}y \\ \text{\textbf{j}}y \\ \text{\textbf{k}}y \end{pmatrix}= \frac{1}{|y|^2} \begin{pmatrix}
		-y_2 & y_1 & -y_4 & y_3 \\
		-y_3 & y_4 & y_1 & -y_2 \\
		-y_4 & -y_3 & y_2 & y_1
	\end{pmatrix}.
\end{aligned}$$
A straightforward computation shows
			$$G^t G = \begin{pmatrix}
				A^t A & 0 \\
				0 & B^t B + C^t C
			\end{pmatrix} = \begin{pmatrix}
				\frac{16|y|^2}{(1+|y|^2)^2 (1+|t|^2)^2} I_2 & 0 \\
				0 & \frac{1}{|y|^4} \left( |y|^2 I_4 - \frac{(1-|y|^2)^2}{(1+|y|^2)^2} y^t y \right)
			\end{pmatrix},$$
			where $G^t$ stands for the transpose of $G$, $I_2$ and $I_4$ are identity matrices of order $2$ and $4$, respectively.
			Consequently, 
			$$\det (G^tG)=\frac{32^2}{|y|^2(1+|y|^2)^6 (1+|t|^2)^4}>0\quad \text{whenever}~ |y|>0.$$
			Evaluating at the point $p=\psi(\textbf{0}, (1, 0, 0, 0))$ shows $\det G|_p=4
			$. By connectedness of the domain, $\det G>0$ everywhere on $S^3_{\delta}\times S^3$. Hence the frames $\{\frac{\partial}{\partial t_1} , \frac{\partial}{\partial t_2}, \frac{\partial}{\partial y_1},\ldots, \frac{\partial}{\partial y_4} \}$ and $\{e_1,\ldots, e_6\}$ define the same orientation.

			Next, using the matrix $G$ introduced above, we express the almost complex structure $J$ from \eqref{J} with respect to the coordinate frame $\{ \frac{\partial}{\partial t_1} , \frac{\partial}{\partial t_2}, \frac{\partial}{\partial y_1},\ldots, \frac{\partial}{\partial y_4}\}$ by
			\begin{equation}\label{GJ0G}
				J(\frac{\partial}{\partial t_1} , \frac{\partial}{\partial t_2}, \frac{\partial}{\partial y_1},\frac{\partial}{\partial y_2},\frac{\partial}{\partial y_3},\frac{\partial}{\partial y_4} ) = (\frac{\partial}{\partial t_1} , \frac{\partial}{\partial t_2}, \frac{\partial}{\partial y_1},\frac{\partial}{\partial y_2},\frac{\partial}{\partial y_3},\frac{\partial}{\partial y_4})  G^{-1}J_0 G.
			\end{equation}
			Thus extending the almost complex structure $J$ from $S^2\times\mathbb{R}^4_{\delta}$ to $S^2\times\mathbb{R}^4$ is equivalent to extending the function matrix $G^{-1}J_0 G$. To accomplish this, we prepare several lemmas in the following.

			The block form of $G$ makes direct computation complicated. To simplify the analysis, we first replace
			$J_0$ with a homotopic almost complex structure in a more convenient form.
			\begin{lem}\label{K}
				$J_0$ is homotopic to $K:=
				\small
				\begin{pmatrix}
					0  & 1 &     &    &     &  \\
					-1 & 0 &    &    &      &  \\
					&    &     &    & -1 & 0 \\
					&    &     &    & 0  & -1 \\
					&    &  1 & 0 &     &  \\
					&   &  0 & 1  &     &
				\end{pmatrix}$ as almost complex structures on $\mathbb{R}^6$. 
			\end{lem}
			\begin{proof}
				
				For $s\in[0, 1]$, define a path of matrices by
				$$E(s) = \small\begin{pmatrix}
					1& &  &  &  &  \\
					&\cos (\frac{\pi}{2}s) & \sin (\frac{\pi}{2}s) & 0 & 0 &   \\
					&0& \cos (\frac{\pi}{2}s) & 0 &   \sin (\frac{\pi}{2}s) & \\
					& - \sin (\frac{\pi}{2}s) & 0 & \cos (\frac{\pi}{2}s) & 0 &  \\
					& 0 & 0 & \sin (\frac{\pi}{2}s) &\cos (\frac{\pi}{2}s) &   \\
					&  &  &  &  & 1					
				\end{pmatrix}.$$
				Clearly, $E(0) = I_6$, and
				$$E(1) = \begin{pmatrix}
					1 &  &  &  &  &  \\
					& 0 & 1 & 0 & 0 &  \\
					& 0 &0 & 0 & 1 & \\
					& -1 & 0 & 0 & 0 &  \\
					& 0 & 0 & 1 & 0 &  \\
					&  &  &  &  & 1
				\end{pmatrix}.$$				
				Since $\det E(s) = \cos^4 (\frac{\pi}{2}s) + \sin^4 (\frac{\pi}{2}s) > 0$ for every $s$, each $E(s)$ is invertible.
				A direct calculation gives
				$$E(0)^{-1} J_0 E(0) = J_0,\quad E(1)^{-1} J_0 E(1) = K.$$
				Hence $s\mapsto E(s)^{-1}J_0E(s)$ is a smooth family of almost complex structures connecting $J_0$ to $K$, i.e., $J_0\simeq K$.
			\end{proof}				
						
				We remark that the upper-left $2\times 2$ block of $K$ cannot be multiplied by $-1$, as doing so would reverse the orientation relative to $J_0$, and thus the resulting structure cannot be homotopic to $J_0$.

			\begin{cor}\label{J0K}
				$G^{-1} J_0 G$ is homotopic to $G^{-1} K G$ as almost complex structures on $S^2\times \mathbb{R}^4_{\delta}$ as defined in \eqref{GJ0G}.
			\end{cor}
			
			The following lemma regarding the commutativity of homotopy groups is a standard result derived from
			\cite{Ste51} (15.3, 16.7, 16.9).
			
			\begin{lem}\label{top gp G}
				Let $Q$ be a connected topological group. For $f,g : S^n \to Q$,
				define $f*g : S^n \to Q$ by $x \mapsto f(x)g(x)$.
				Then $f*g \simeq g*f : S^n \to Q$.
			\end{lem}

			In particular, taking $Q=GL_+(m, \mathbb{R})$, the general linear group of matrices with positive determinant, we obtain
			
			\begin{cor}\label{ABsimeqBA}
				$AB \simeq BA$ for $A, B: S^n\to GL_+(m, \mathbb{R})$.
			\end{cor}
			Since the matrix $P = \text{diag}(1, -1, 1, 1)$ has $\det P=-1<0$, we cannot apply Corollary \ref{ABsimeqBA} directly. However, we still have 
			\begin{lem}\label{PfP}
				Let $f : S^3 \to SO(4)$ be defined by $y \mapsto (y^t, (\mathbf{i}y)^t, (\mathbf{j}y)^t, (\mathbf{k}y)^t)$, and let	$P = \text{diag}(1, -1, 1, 1)$. Then the homotopy class $[P^{-1}fP]$ is equal to $[f]$ in $\pi_3SO(6)$, where we employ the homomorphism $\pi_3SO(4) \rightarrow \pi_3SO(6)$ induced by the standard inclusion $SO(4) \hookrightarrow SO(6)$.
			\end{lem}			
			\begin{proof}
				It is well known that the special orthogonal group $SO(4)$ admits a fiber bundle structure over $S^3$ with structure group and typical fiber $SO(3)$. This bundle is isomorphic to the trivial bundle $S^3\times SO(3)$. Consequently, we have the following isomorphism of homotopy groups:
				$$\pi_3 SO(4)\cong \pi_3S^3\oplus \pi_3SO(3)\cong \mathbb{Z}\oplus \mathbb{Z}.$$
				
				Observe that the homomorphism $\rho: S^3\rightarrow SO(3)$ generates the infinite cyclic group $\pi_3 SO(3)\cong \mathbb{Z}$, where $\rho(y)$ is given by
				$$
				\begin{aligned}
					\rho(y): ~\mathbb{R}^3&\rightarrow\mathbb{R}^3\\
					x&\mapsto \rho(y)(x):=yx\overline{y},
				\end{aligned}$$
				for $x\in \mathbb{R}^3\cong\text{Im}\mathbb{H}$. 
				In fact, $\rho$ realizes the standard double covering map from $S^3$ to $\mathbb{R}P^3\cong SO(3)$.
				Clearly, the map $\rho(y)$ extends naturally to an automorphism of $\mathbb{R}^4$ since $y\textbf{1}\overline{y}=\textbf{1}$ for all $y\in S^3\subset\mathbb{H}$. We retain the notation
				$\rho(y)$ for this extension. 

				Let $\sigma: S^3\rightarrow SO(4)$ be defined by left multiplication
				$$\begin{aligned}
					\sigma(y):~\mathbb{R}^4&\rightarrow\mathbb{R}^4\\
					x&\mapsto \sigma(y)x:=yx
				\end{aligned}$$
				for $x\in\mathbb{R}^4\cong\mathbb{H}$. 
				Then the homomorphism $\sigma$ induces an isomorphism of $\pi_3S^3$ onto a subgroup of $\pi_3SO(4)$. 
				
				
				Let $\alpha_0\in \pi_0O(4)\cong\mathbb{Z}_2$ denote the nontrivial element. The action of $\alpha_0$ on $\pi_3SO(4)$ is given by
				$\alpha_0([\phi]):=[P^{-1}\phi P]$, where $P=\text{diag} (1, -1, 1, 1)$ for $\phi: S^3\rightarrow SO(4)$.
				By \S 22.5 in \cite{Ste51}, this action acts trivially on the image of
				$\pi_3SO(3)$ in $\pi_3SO(4)$, so that $\alpha_0 ([\rho])=[\rho]$. Furthermore, formula (15) in \S 22.7 yields the identity
				$$\alpha_0 ([\sigma])=[P^{-1}\sigma P]=\alpha_0 ([\rho])-[\sigma]=[\rho]-[\sigma]~\quad \text{in}~\pi_3SO(4).$$
				By \S 23.6 in \cite{Ste51}, the inclusion-induced homomorphism $\pi_3SO(4)\rightarrow\pi_3SO(6)$ satisfies the relation $[\rho]=2[\sigma]$ in $\pi_3SO(6)$. Consequently, we obtain
				$$\alpha_0 ([\sigma])=[\sigma]~\quad \text{in}~\pi_3SO(6).$$

				Notice that the map $f : S^3 \to SO(4)$ 
				in Lemma \ref{PfP} is actually a homomorphism defined by right multiplication:
				$$\begin{aligned}
					f(y):~\mathbb{R}^4&\rightarrow\mathbb{R}^4\\
					x&\mapsto f(y)x:=x\overline{y}
				\end{aligned}$$
				for $y\in S^3\subset \mathbb{H}$ and $x\in\mathbb{R}^4\cong\mathbb{H}$. 
				It is direct to see that 
				$\rho(y)=\sigma(y)\circ f(y)$, 
				and thus
					$$[\rho]=[\sigma]+[f]~\quad \text{in}~\pi_3SO(4).$$
					Applying the action of $\alpha_0$ to both sides of the identity, we obtain
					$$[P^{-1}f P] =\alpha_0([f])=\alpha_0([\rho]-[\sigma])=[\sigma]=[f]\quad \text{in}~\pi_3SO(6).$$

				\end{proof}
				
				Let $\widetilde{D}=\left(y^t, (\mathbf{i}y)^t, (\mathbf{j}y)^t, (\mathbf{k}y)^t\right)$ as in Lemma \ref{PfP}. It follows from Lemma \ref{PfP} that $\widetilde{D}\simeq P^{-1}\widetilde{D}P$ as maps from $S^3$ to $SO(4)$ with
				$$D:= P^{-1}\widetilde{D}P=\left( \begin{array}{cccc}
					y_1 & y_2 & -y_3 & -y_4 \\
					-y_2 & y_1 & -y_4 & y_3 \\
					y_3 & y_4 & y_1 & y_2 \\
					y_4 & -y_3 & -y_2 & y_1
				\end{array}
				\right).$$
				Notice that  $D$ is the real representation of the complex matrix
				$$D' := \begin{pmatrix}
					y_1 & y_2 \\
					-y_2& y_1
				\end{pmatrix}+\sqrt{-1}  \begin{pmatrix}
					y_3 & y_4 \\
					y_4& -y_3
				\end{pmatrix}.$$
				Thus $\det\widetilde{D}=\det D=|\det D'|^2=|y|^4>0$ if $|y|>0$.
				Moreover, we observe that the matrix $K$ in Lemma \ref{K} satisfies
				$$
				\begin{pmatrix}
					I_2 &  \\
					&D 	\end{pmatrix}  K = K\begin{pmatrix}
					I_2 &  \\
					&D 	\end{pmatrix}.
				$$


				Now, we are ready to prove the extendability of $G^{-1}J_0G$ from $S^2\times\mathbb{R}^4_{\delta}$
				to $S^2\times \mathbb{R}^4$. Combining Corollaries \ref{J0K}, \ref{ABsimeqBA} and Lemma \ref{PfP} together with the facts $\det G>0$ and $\det \widetilde{D}>0$ for $|y|>0$, we obtain
				\begin{eqnarray}
					G^{-1}J_0G&\simeq& G^{-1}KG=G^{-1}\begin{pmatrix}
						I_2 &  \\
						&D 	\end{pmatrix}^{-1}K \begin{pmatrix}
						I_2 &  \\
						&D 	\end{pmatrix}G\nonumber\\
					&=&\left( \begin{pmatrix}
						I_2 &  \\
						&D 	\end{pmatrix}G\right)^{-1}K\left(\begin{pmatrix}
						I_2 &  \\
						&D 	\end{pmatrix}G\right)\label{homotopy}\\
					&\simeq&\left( \begin{pmatrix}
						I_2 &  \\
						&\widetilde{D}	\end{pmatrix}G\right)^{-1}K\left(\begin{pmatrix}
						I_2 &  \\
						&\widetilde{D}	\end{pmatrix}G\right)\nonumber\\
					&\simeq&\left( G\begin{pmatrix}
						I_2 &  \\
						&\widetilde{D}	\end{pmatrix}\right)^{-1}K\left(G\begin{pmatrix}
						I_2 &  \\
						&\widetilde{D}	\end{pmatrix}\right).\nonumber
				\end{eqnarray}	
				For the matrix $G$ appearing in \eqref{G}, one checks directly that
				$$
				G \begin{pmatrix} I_2 &\\ &\widetilde{D}  \end{pmatrix} = \left( \begin{array}{cc} H& \\ &I_3 \end{array} \right),
				$$				
				where $H=(h_1, h_2, h_3)$ is a $3\times 3$ matrix with column vectors $h_1, h_2, h_3$ with $A_{3\times 2}=\begin{pmatrix}
					h_1& h_2
				\end{pmatrix}$ as given in \eqref{G}. Explicitly,
				\begin{eqnarray*}
					h_1 &=& \frac{4|y|}{(1+|y|^2)^2(1+|t|^2)^2} \left((1-|y|^2) \begin{pmatrix}
						-2t_1 \\
						1-t_1^2+t_2^2 \\
						-2t_1 t_2
					\end{pmatrix} + 2|y| \begin{pmatrix}
						-2t_2 \\
						-2t_1 t_2 \\
						1+t_1^2-t_2^2
					\end{pmatrix} \right),\\
					h_2 &=& \frac{4|y|}{(1+|y|^2)^2(1+|t|^2)^2} \left((1-|y|^2) \begin{pmatrix}
						-2t_2 \\
						-2t_1 t_2 \\
						1+t_1^2-t_2^2
					\end{pmatrix} - 2|y| \begin{pmatrix}
						-2t_1 \\
						1-t_1^2+t_2^2 \\
						-2t_1t_2
					\end{pmatrix} \right), \\
					h_3&=&\frac{2|y|}{(1+|y|^2)(1+|t|^2)}\begin{pmatrix}
						1-|t|^2 \\
						2t_1 \\
						2t_2
					\end{pmatrix}.
				\end{eqnarray*}
				Notice that the three vectors $h_1, h_2, h_3$ are pairwise orthogonal, and
				$$|h_1|^2 = |h_2|^2 = \frac{16|y|^2}{(1+|y|^2)^2(1+|t|^2)^2}, \quad |h_3|^2 = \frac{4|y|^2}{(1+|y|^2)^2}.$$
				Express $H$ as $H=|y|\widetilde{H}$ with $\det \widetilde{H} = \frac{32}{(1+|y|^2)^3 (1+|t|^2)^3}$. Obviously, $\widetilde{H}$ is nowhere singular, even at $y=0$. Therefore, using
				$|y|$ as the homotopy parameter, we obtain a homotopy $H\simeq I_3$ for sufficiently small $|y|>0$. 
				Consequently, $G^{-1}J_0G$ can be extended to $\left(S^2\backslash\{S\}\right)\times\mathbb{R}^4$.
				
				From the arguments above, we see that the extendability does not depend on $t$, so replacing $\phi_S$ with $\phi_N$ in \eqref{phi_S} does not affect the result. 
				Hence $G^{-1}J_0G$ can be extended to $S^2\times\mathbb{R}^4$.
				Therefore, the almost complex structure $J$ can be extended smoothly to the entire $S^6$, as we hoped.
					\hfill$\Box$
				
			 \begin{rem}
			 	 The map $G$ admits no extension to $S^2\times\mathbb{R}^4$ because it is not null-homotopic. More precisely, this conclusion is deduced from the following key observation: the modified map $\begin{pmatrix}
			 	 	I_2&\\
			 	 	& D
			 	 \end{pmatrix}G$ 
			 	 is null-homotopic (as shown in the proof above), while the map $D$ itself is not null-homotopic. In fact, $D$ is the real representation of the map $D': S^3\rightarrow SU(2)$ defined by $D'(z, w)=\begin{pmatrix}
			 	 	z&w\\
			 	 	-\overline{w}& \overline{z}
			 	 \end{pmatrix}$, where $z=y_1+\sqrt{-1}y_3, w=y_2+\sqrt{-1}y_4$.  Notably, this map $D'$ represents a generator of the homotopy group $\pi_3SU(2)\cong\mathbb{Z}$.
			 \end{rem}
		 

				\section{\textbf{Non-extendability of $J$ to an integrable almost complex structure on $S^6$}}
				
				This section is devoted to proving the non-extendability of the complex structure $J$ on $S^3_{\delta}\times S^3$ to an integrable almost complex structure on $S^6$.
				
				It is well known that there exist no non-constant holomorphic functions on compact complex manifolds. So one considers only meromorphic functions. 
				Recall that the algebraic dimension $a(X)$ of a compact complex manifold $X$ is defined to be the transcendence degree of its field of meromorphic functions over $\mathbb{C}$. Thus $a(X)=0$ if and only if every meromorphic function on $X$ is constant. A surprising result of Campana, Demailly and Peternell (\cite{CDP98, CDP20}) states:
				\vspace{2mm}
				
				\noindent
				\textbf{Theorem (Campana-Demailly-Peternell)}. \textit{Let $X$ be a compact complex manifold homeomorphic to $S^6$. Then $a(X) = 0$.}
				\vspace{2mm}
				
				Hence every meromorphic function on a hypothetical complex manifold $S^6$  must be constant.
				For compact complex manifolds with no non-constant meromorphic functions, Krasnov \cite{Kra75} established a strong estimate:
				\vspace{2mm}
				
				\noindent
\textbf{Theorem (Krasnov).} \textit{Let $X$ be a connected compact complex manifold of complex dimension 
	$n$ with no non-constant meromorphic functions. Then the number of irreducible hypersurfaces in $X$ is finite, and not greater than $n+h^{1,1}-h^{1,0}$, where $h^{p,q}=\dim H^q(X, \Omega_X^p)$ is the Hodge number of the complex manifold $X$, and $ \Omega_X^p$ is the sheaf of holomorphic $p$-forms.}
\vspace{2mm}

				Suppose now that the complex structure $J$ we defined on $S^3_{\delta} \times S^3$ could be extended smoothly to a complex structure on $S^6$. Denote by $S^6_J$ the resulting compact complex manifold. 
				From Krasnov's theorem, it follows that $S^6_J$ can contain only finitely many closed complex hypersurfaces.

				However, $S^3_{\delta} \times S^3\subset S^6_J$ already contains infinitely many distinct closed complex hypersurfaces. To see this, 
				consider the free $S^1$-action on $S^3\subset\mathbb{C}^2$:
				\begin{equation*}
					\begin{aligned}
						S^1\times S^3~&\rightarrow S^3\\
						e^{i\theta}, (z, w)&\mapsto e^{i\theta}(z, w):=(e^{i\theta}z, e^{i\theta}w).
					\end{aligned}
				\end{equation*}
				Its orbit space is $S^3\big/S^1\cong \mathbb{C}P^1\cong S^2$.
				For a base point $(z_0, w_0)\in S^3$, we call its orbit $\{e^{i\theta}(z_0, w_0)~|~\theta\in\mathbb{R}\}$ a \textit{standard circle} in $S^3$.
				In this manner, the Hopf manifold 
				\begin{equation}\label{cplx hyp}
					S^1\times S^3:=\left\{\left( e^{i\theta}(z_0, w_0), ~v\right)~|~\theta\in\mathbb{R}, v\in S^3\right\}
				\end{equation}
				is a complex hypersurface in $S^3\times S^3$. 

				Observe that for every $\varphi$ with $\cos\varphi\in(0, \delta)$, the point $p_{\varphi}=(\cos\varphi, 0, \sin\varphi, 0)$ lies in $S^3_{\delta}$, and its entire orbit
				$$\mathcal{C}_{p_{\varphi}}=\left\{ e^{i\theta}p_{\varphi}~|~\theta\in\mathbb{R}\right\}$$
				is a standard circle contained in $S^3_{\delta}$. Since there are infinitely many such standard circles, 
				we obtain infinitely many distinct closed complex hypersurfaces $\mathcal{C_{p_{\varphi}}}\times S^3$ in $S^3_{\delta}\times S^3\subset S^6_J$.
				This contradicts Krasnov's finiteness theorem. 
				
				Consequently, the complex structure $J$ on $S^3_{\delta} \times S^3$ cannot be extended to an integrable almost complex structure on $S^6$, as we claimed.
					\hfill $\Box$
				
				\begin{rem}\label{rem3.1}
				 For $\delta\in(0, 1)$, define $B^3_{\delta}:=\{x=(x_1, x_2, x_3, x_4)\in S^3~|~-\delta<x_1<\delta\}$. 
				 Observe that the previously constructed $\mathcal{C}_{p_{\varphi}} \times S^3$ is actually contained in $B^3_\delta \times S^3$.
				 As $\delta$ approaches $0$, the measure of $B^3_\delta$ (and hence of $B^3_\delta \times S^3$) becomes arbitrarily small, while $B^3_\delta \times S^3$ remains a connected open subset of $S^6$.
				 Therefore, even if we replace $S^3_\delta \times S^3$ with such an open subset $B^3_\delta \times S^3$ of arbitrarily small measure, the conclusion of Theorem \ref{thm} continues to hold, as noted in Remark \ref{rem1.1}.
				\end{rem}

				\section{\textbf{The proof of Proposition \ref{prop}}}
				
We first show that $M$ admits an almost complex structure as we desired. 	
				Recall that for an oriented smooth manifold $M^{2n}$, all the obstructions to the existence of an almost complex structure inducing the given orientation lie in the cohomology groups
				\begin{equation}\label{obs}
				H^{k+1}(M, \pi_k(\mathrm{SO}(2n)/\mathrm{U}(n))), \qquad k = 1, 2, \ldots, 2n-1.
				\end{equation}
				In our situation, $n=2$ and $\mathrm{SO}(4)/\mathrm{U}(2)\cong S^2$. 
				Since $\pi_1(S^2) = 0$, it suffices to examine the groups $H^{k+1}(M, \pi_k(S^2))$ for $k=2, 3$. It is known that there exists an isoparametric foliation on $\mathbb{C}P^n$ with two focal submanifolds---$\mathbb{C}P^{n-1}$ and a point $\{p\}$ (see, for example, \cite{GTY15}).
				Thus, $\mathbb{C}P^2$ admits an isoparametric foliation with focal submanifolds $S^2$ and a point $\{p\}$. According to \cite{GT13, GT14}, analogous to the case of spheres, there also exists a decomposition of a Riemannian manifold into two disc bundles over the two focal submanifolds. Taking $D^4$ as a closed disc neighborhood of $p$, we see that $M=\mathbb{C}P^2 \setminus D^4$ is homeomorphic to the total space of a real rank-$2$ vector bundle over $S^2$. Consequently, $M\simeq S^2$. Therefore, all the relevant cohomology groups vanish, and $M$ admits an almost complex structure compatible with any given orientation.

				Let us now fix an almost complex structure $J_0$ on $M$ that induces the negative orientation.
				By the Gromov-Landweber theorem, 
				$J_0$ can be deformed to a complex structure $J$ on $M$, which also induces the negative orientation.
				Suppose that $J$ could be extended smoothly to an almost complex structure on $\mathbb{C}P^2$.
				The extension would then endow $\mathbb{C}P^2$ with the negative orientation, making it $\overline{CP^{2}}$. However, the counter-intuitive results of Kahn \cite{Kah69} and Tang-Zhang \cite{TZ02} assert that
				$\overline{CP^{2n}}$  admits no almost complex structure, a contradiction.  
\hfill$\Box$
\begin{rem}
	\label{rem:s6_uniqueness}
	While \eqref{obs} concerns existence, the uniqueness of almost complex structures on $S^6$ is governed by difference classes in $H^k(S^6; \pi_k(\mathbb{C}P^3))$, identifying the fiber $SO(6)/U(3)$ with $\mathbb{C}P^3$. These classes vanish for all $1 \le k \le 6$: trivially for $k < 6$ due to the base cohomology, and for $k=6$ since $\pi_6(\mathbb{C}P^3) = 0$. Consequently, the space of almost complex structures compatible with a given orientation is connected.
\end{rem}

	\begin{ack}
The authors would like to express their sincere gratitude to Professors Yum-Tong Siu, Fangyang Zheng, and Kefeng Liu for their insightful comments and valuable suggestions.
	\end{ack}


\begin{thebibliography}{123}
				
				\bibitem[Ada79]{Ada79}
				Adachi M., \textit{Construction of complex structures on open manifolds},
				Proc. Japan Acad. Ser. A, \textbf{55}(1979), 222--224.
				
				
				\bibitem[BV68]{BV68}
				Brieskorn E., Van de Ven A., \emph{Some complex structures on products of homotopy spheres}, Topology, \textbf{7}(1968), 389--393.
				
				\bibitem[CDP98]{CDP98}
				Campana F., Demailly J.-P., Peternell T., \emph{The algebraic dimension of compact complex threefolds with vanishing second Betti number}, Compositio Math. \textbf{112}(1998), 77--91.
				
				\bibitem[CDP20]{CDP20}
				Campana F., Demailly J.-P., Peternell T., \emph{The algebraic dimension of compact complex threefolds with vanishing second Betti number}, Compositio Math., \textbf{156}(2020), 679--696.
				
				
				\bibitem[CE53]{CE53}
				Calabi E., Eckmann B., \emph{A class of compact, complex manifolds which are not algebraic},
				Ann. Math., \textbf{58}(1953), 494--500.
				
				
				\bibitem[EF51]{EF51}
				Eckmann B., Fr\"olicher A., \emph{Sur l'int\'egrabilit\'e des structures presque complexes}, C. R. Acad. Sci. Paris \textbf{232}(1951), 2284--2286.
				
				
				
				\bibitem[EL51]{EL51}
				Ehresmann C., Libermann P., \emph{Sur les structures presque hermitiennes isotropes}, C. R. Acad. Sci. Paris, \textbf{232}(1951) 1281--1283.
				
				
				
				\bibitem[Gro73]{Gro73}
				Gromov M. L., \emph{Degenerate smooth mappings}. Math. Note USSR, \textbf{14}(1973), 849--853.
				
				\bibitem[GT13]{GT13}
				Ge J. Q., Tang Z. Z., \textit{Isoparametric functions and exotic spheres}, J. Reine Angew. Math., \textbf{683}(2013), 161--180.
				
				\bibitem[GT14]{GT14}
				Ge J. Q., Tang Z. Z., \textit{Geometry of isoparametric hypersurfaces in Riemannian manifolds}. Asian J. Math, \textbf{18}(2014), 117--126.
				
				\bibitem[GTY15]{GTY15}
				Ge J. Q., Tang Z. Z., Yan W. J., \textit{A filtration for isoparametric hypersurfaces in Riemannian Manifolds}, J. Math. Soc. Japan, \textbf{67}(2015), 1179--1212.
				
				
				\bibitem[Huy05]{Huy05}
				Huybrechts D., \emph{Complex geometry}, Universitext, Springer-Verlag, Berlin, 2005, xii+309 pp.
				
				\bibitem[Kah69]{Kah69}
				Kahn P. J., \textit{Obstructions to extending almost X-structures}, Illinois J. Math. \textbf{13}(1969), 336--357.
				
				\bibitem[Kir47]{Kir47}
				Kirchhoff A., \emph{Sur l'existence des certains champs tensoriels sur les sph\`eres \`a n dimensions}, C. R. Acad. Sci. Paris \textbf{225}(1947) 1258--1260.
				
				\bibitem[Kra75]{Kra75}
				Krasnov V. A., \emph{Compact complex manifolds without meromorphic functions} (in Russian), Mat. Zametki \textbf{17}(1975) 119--122; English translation in Math. Notes \textbf{17}(1975), 69--71.
				
				
				\bibitem[Lan74]{Lan74}
				Landweber P. S., \emph{Complex structures on open manifolds}, Topology, \textbf{13}(1974), 69--75.
				
				
				\bibitem[M\"un80]{Mun80}
				M\"unzner H.F., \textit{Isoparametrische Hyperfl\"achen in Sphären, I, II}, Math. Ann. \textbf{251}(1980), 57--71. \textbf{256}(1981), 215--232.
				
				
				\bibitem[QTY25]{QTY25}
				Qian C., Tang Z. Z., Yan W. J., \emph{Isoparametric foliations and complex structures}, arXiv:2502.09041.
				
				\bibitem[Ste51]{Ste51}
				Steenrod N., \emph{The Topology of Fibre Bundles}, Princeton Mathematical Series, vol. 14. Princeton University Press, Princeton, N. J., 1951, viii+224 pp.
				
				
				\bibitem[SY94]{SY94}
				Schoen R., Yau S.-T., \emph{Open problems in geometry,  Lectures on differential geometry}, Conference Proceedings and Lecture Notes in Geometry and Topology, Volume 1, International Press, Cambridge, MA, 1994, v+235 pp.
				
				
				\bibitem[TY22]{TY22}
				Tang Z. Z., Yan W. J., \emph{Isoparametric hypersurfaces and complex structures}, Acta Math. Scientia Ser. B, \textbf{42}(2022), 2223--2229.
				
				
				\bibitem[TZ02]{TZ02}
				Tang Z. Z., Zhang W. P., \emph{A generalization of the Atiyah-Dupont vector fields theory}, Commun. Contemp. Math. \textbf{4}(2002),  777--796.
				
				\bibitem[VdV66]{VdV66}
				Van de Ven A., \emph{On the Chern numbers of certain complex and almost complex manifolds}, Proc.
				Nat. Acad. Sci. U. S. A. \textbf{55}(1966), 1624--1627.
				
				\bibitem[Yau76]{Yau76}
				Yau S.-T., Parallelizable manifolds without complex structure, Topology, \textbf{15}(1976), 51--53.
				
				
				
				
				
				
				
				
				
				
				
				
				
				
				
				
				
				
				
				
				
				
				
				
				
				
				
				
				
				
				
				

			\end{thebibliography}
		\end{document}